\documentclass[11pt]{amsart}
\usepackage{amssymb}
\usepackage{amsmath}
\usepackage[active]{srcltx}
\usepackage[latin2]{inputenc}
\usepackage{t1enc}
\usepackage{verbatim}
\usepackage{amsmath,amsfonts,amssymb,amsthm}
\usepackage[mathcal]{eucal}
\usepackage{enumerate}
\usepackage[centertags]{amsmath}
\usepackage{graphics}

\newtheorem{theorem}{Theorem}

\newtheorem{corollary}{Corollary}

\begin{document}
\author{L. E. Persson and G. Tephnadze}
\title[Fejér means]{A note on Vilenkin-Fejér means on the Martingale Hardy
spaces $H_{p}$}
\address{L.-E. Persson, Department of Engineering Sciences and Mathematics,
Lule\aa\ University of Technology, SE-971 87 Lule\aa , Sweden and Narvik
University College, P.O. Box 385, N-8505, Narvik, Norway.}
\email{larserik@ltu.se}
\address{G. Tephnadze, Department of Mathematics, Faculty of Exact and
Natural Sciences, Tbilisi State University, Chavchavadze str. 1, Tbilisi
0128, Georgia}
\email{giorgitephnadze@gmail.com}
\thanks{The research was supported by Shota Rustaveli National Science
Foundation grant no.52/54 (Bounded operators on the martingale Hardy spaces).%
}
\date{}
\maketitle

\begin{abstract}
The main aim of this note is to derive necessary and sufficient conditions
for the convergence of Fejér means in terms of the modulus of continuity of
the Hardy spaces $H_{p},$ $\left( 0<p\leq 1\right) $.
\end{abstract}

\date{}



\section{Introduction and Preliminary Results}

Let $\mathbb{P}_{+}$ denote the set of the positive integers and $\mathbb{P}%
:=\mathbb{P}_{+}\cup \{0\}.$

Let $m:=(m_{0,}m_{1},\dots)$ denote a sequence of the positive integers not
less than 2.

Denote by
\begin{equation*}
Z_{m_{k}}:=\{0,1,\dots ,m_{k}-1\}
\end{equation*}
the additive group of integers modulo $m_{k}.$

Define the group $G_{m}$ as the complete direct product of the group $%
Z_{m_{j}}$ with the product of the discrete topologies of $Z_{m_{j}}$`s.

The direct product $\mu $ of the measures
\begin{equation*}
\mu _{k}\left( \{j\}\right) :=1/m_{k},(j\in Z_{m_{k}})
\end{equation*}
is the Haar measure on $G_{m_{\text{ }}}$with $\mu \left( G_{m}\right) =1.$

\textbf{In this paper we consider bounded Vilenkin groups only, which are
defined by the condition }$\sup_{n}m_{n}<\infty .$

The elements of $G_{m}$ are represented by sequences
\begin{equation*}
x:=(x_{0},x_{1},\dots ,x_{k},\dots )\qquad \left( \text{ }x_{k}\in
Z_{m_{k}}\right) .
\end{equation*}

It is easy to give a base for the neighbourhoods of $G_{m}:$
\begin{equation*}
I_{0}\left( x\right) :=G_{m},\text{ \ }I_{n}(x):=\{y\in G_{m}\mid
y_{0}=x_{0},\dots ,y_{n-1}=x_{n-1}\}\text{ }(x\in G_{m},\text{ }n\in \mathbb{%
P}).
\end{equation*}%
Denote $I_{n}:=I_{n}\left( 0\right) $ and $\overline{I_{n}}:=G_{m}$ $%
\backslash $ $I_{n},$ for $n\in \mathbb{P}$. Let $e_{n}:=\left( 0,\dots
,0,x_{n}=1,0,\dots \right) \in G_{m},$ $\left( n\in \mathbb{P}\right) .$

The norm (or quasi-norm) of the space $L_{p}(G_{m})$ is defined by \qquad
\qquad \thinspace\
\begin{equation*}
\left\Vert f\right\Vert _{p}:=\left( \int_{G_{m}}\left\vert f(x)\right\vert
^{p}d\mu (x)\right) ^{1/p}\qquad \left( 0<p<\infty \right) .
\end{equation*}

The space $weak-L_{p}\left( G_{m}\right) $ consists of all measurable
functions $f,$ for which
\begin{equation*}
\left\Vert f\right\Vert _{weak-L_{p}}^{p}:=\underset{\lambda >0}{\sup }%
\lambda ^{p}\mu \left( f>\lambda \right) <+\infty .
\end{equation*}

If we define the so-called generalized number system based on $m$ in the
following way:
\begin{equation*}
M_{0}:=1,\text{ \qquad }M_{k+1}:=m_{k}M_{k\text{ }}\ \qquad (k\in \mathbb{P}%
),
\end{equation*}%
then every $n\in \mathbb{P}$ can be uniquely expressed as $%
n=\sum_{j=0}^{\infty }n_{j}M_{j},$ where $n_{j}\in Z_{m_{j}}$ $~(j\in
\mathbb{P})$ and only a finite number of $n_{j}`$s differ from zero. Let $%
\left\vert n\right\vert :=\max $ $\{j\in \mathbb{P};$ $n_{j}\neq 0\}.$

Next, we define the complex valued function $r_{k}\left( x\right)
:G_{m}\rightarrow \mathbb{C},$ called the generalized Rademacher functions
in the following way:
\begin{equation*}
r_{k}\left( x\right) :=\exp \left( 2\pi \imath x_{k}/m_{k}\right) \text{
\qquad }\left( \imath ^{2}=-1,\text{ }x\in G_{m},\text{ }k\in \mathbb{P}%
\right) .
\end{equation*}

Moreover, the Vilenkin system $\psi :=(\psi _{n}:n\in \mathbb{P})$ on $G_{m}$
is defined as follows:
\begin{equation*}
\psi _{n}:=\prod_{k=0}^{\infty }r_{k}^{n_{k}}\left( x\right) \text{ \qquad }%
\left( n\in \mathbb{P}\right) .
\end{equation*}

In particular, we call this system the Walsh-Paley one when $m\equiv 2.$ It
is known that the Vilenkin system is orthonormal and complete in $%
L_{2}\left( G_{m}\right) \,$(see e.g. \cite{AVD,Vi}).

Hence we can introduce analogues of the usual definitions in
Fourier-analysis. If $f\in L_{1}\left( G_{m}\right) $ we can define the
Fourier coefficients, the partial sums of the Fourier series, the Fejér
means, the Dirichlet and Fejér kernels with respect to the Vilenkin system
in the usual manner:
\begin{eqnarray*}
\widehat{f}\left( n\right) &:&=\int_{G_{m}}f\overline{\psi }_{n}d\mu ,\text{%
\ \ \ \ }S_{n}f:=\sum_{k=0}^{n-1}\widehat{f}\left( k\right) \psi _{k},\text{
\ \ \ }\sigma _{n}f:=\frac{1}{n}\sum_{k=1}^{n}S_{k}f, \\
D_{n} &:&=\sum_{k=0}^{n-1}\psi _{k},\text{ \ \ \ \ \ \ \ \ \ \ }K_{n}:=\frac{%
1}{n}\overset{n}{\underset{k=1}{\sum }}D_{k}\text{ \ \thinspace\ }\left(
n\in \mathbb{P}_{+}\right) .
\end{eqnarray*}

Recall that
\begin{equation}
\quad \hspace*{0in}D_{M_{n}}\left( x\right) =\left\{
\begin{array}{ll}
M_{n}, & \text{if\thinspace \thinspace }x\in I_{n}, \\
0, & \text{\thinspace if \thinspace \thinspace }x\notin I_{n}.%
\end{array}%
\right.  \label{2}
\end{equation}%
\vspace{0pt}

The $\sigma $-algebra generated by the intervals $\left\{ I_{n}\left(
x\right) :x\in G_{m}\right\} $ will be denoted by $\digamma _{n}$ $\left(
n\in \mathbb{P}\right) .$ Denote by $f=\left( f^{\left( n\right) },n\in
\mathbb{P}\right) $ the martingale with respect to $\digamma _{n}$ $\left(
n\in \mathbb{P}\right) $ (for details see e.g. \cite{We1}). The maximal
function of the martingale $f$ is defined by \qquad
\begin{equation*}
f^{\ast }\left( x\right) =\sup_{n\in \mathbb{P}}\left\vert f^{\left(
n\right) }\left( x\right) \right\vert .
\end{equation*}

In the case $f\in L_{1}(G_{m}),$ the maximal functions can also be given by
\begin{equation*}
f^{\ast }\left( x\right) =\sup_{n\in \mathbb{P}}\frac{1}{\left\vert
I_{n}\left( x\right) \right\vert }\left\vert \int_{I_{n}\left( x\right)
}f\left( u\right) d\mu \left( u\right) \right\vert
\end{equation*}

For $0<p<\infty $ the Hardy martingale spaces $H_{p}$ $\left( G_{m}\right) $
consist of all martingales for which
\begin{equation}
\left\Vert f\right\Vert _{H_{p}}:=\left\Vert f^{\ast }\right\Vert
_{p}<\infty .  \label{maxfn}
\end{equation}

If $f\in L_{1}(G_{m}),$ then it is easy to show that the sequence $\left(
S_{M_{n}}\left( f\right) :n\in \mathbb{P}\right) $ is a martingale. If $%
f=\left( f^{\left( n\right) },n\in \mathbb{P}\right) $ is a martingale, then
the Vilenkin-Fourier coefficients must be defined in a slightly different
manner: $\qquad \qquad $
\begin{equation*}
\widehat{f}\left( i\right) :=\lim_{k\rightarrow \infty
}\int_{G_{m}}f^{\left( k\right) }\left( x\right) \overline{\psi }_{i}\left(
x\right) d\mu \left( x\right) .
\end{equation*}

The Vilenkin-Fourier coefficients of $f\in L_{1}\left( G_{m}\right) $ are
the same as those of the martingale $\left( S_{M_{n}}\left( f\right) :n\in
\mathbb{P}\right) $ obtained from $f$.

For the martingale $f$ \ we consider the following maximal operators:%
\begin{equation*}
\sigma ^{\ast }f:=\sup_{n\in \mathbb{P}}\left\vert \sigma _{n}f\right\vert ,%
\text{\ \ \ }\sigma ^{\#}f:=\sup_{n\in \mathbb{P}}\left\vert \sigma
_{M_{n}}f\right\vert ,\text{ \ \ }\widetilde{\sigma }_{p}^{\ast
}:=\sup_{n\in \mathbb{P}_{+}}\frac{\left\vert \sigma _{n}\right\vert }{%
n^{1/p-2}\log ^{2\left[ 1/2+p\right] }\left( n+1\right) },\text{\ }
\end{equation*}%
where $0<p\leq 1/2$ and $\left[ 1/2+p\right] $ denotes the integer part of $%
1/2+p.$

A weak type-$\left( 1,1\right) $ inequality for the maximal operator of Fejé%
r means $\sigma ^{\ast }$ can be found in Schipp \cite{Sc} for Walsh series
and in Pál, Simon \cite{PS} for bounded Vilenkin series. Fujji \cite{Fu} and
Simon \cite{Si2} verified that $\sigma ^{\ast }$ is bounded from $H_{1}$ to $%
L_{1}$.

Weisz \cite{We2} generalized this result and proved the following:

\textbf{Theorem W1 (Weisz). }The maximal operator $\sigma ^{\ast }$ \ is
bounded from the martingale space $H_{p}$ to the space $L_{p}$ for $p>1/2$.

Simon \cite{Si1} gave a counterexample, which shows that boundedness does
not hold for $0<p<1/2.$ The counterexample for $p=1/2$ is due to Goginava
\cite{Go}, (see also \cite{BGG2}). Weisz \cite{we4} proved that $\sigma
^{\ast }$ is bounded from the Hardy space $H_{1/2}$ to the space $%
L_{weak-1/2}$.

In \cite{tep2} and \cite{tep3} (for Walsh system see \cite{GoSzeged}) it was
proved that the maximal operators $\widetilde{\sigma }_{p}^{\ast }$ with
respect to Vilenkin systems, where $0<p\leq 1/2$ and $\left[ 1/2+p\right] $
denotes the integer part of $1/2+p,$ is bounded from the Hardy space $H_{p}$
to the space $L_{p}.$ Moreover, we showed that the order of deviant
behaviour of the $n$-th Fejér means was given exactly. As a corollary it was
pointed out that
\begin{equation}
\left\Vert \sigma _{n}f\right\Vert _{p}\leq c_{p}n^{1/p-2}\log ^{2\left[
1/2+p\right] }n\left\Vert f\right\Vert _{H_{p}},\text{ \ \ }\left(
n=2,3,....\right) .  \label{1}
\end{equation}

Weisz \cite{we2} also proved that the following is true:

\textbf{Theorem W2 (Weisz). }The maximal operator $\sigma ^{\#}f$ \ is
bounded from the martingale Hardy space $H_{p}$ $\left( G_{m}\right) $ to
the space $L_{p}$ $\left( G_{m}\right) $ for $p>0.$

Moreover, he also considered the norm convergence of Fejér means of
Vilenkin-Fourier series and proved the following:

\textbf{Theorem W3 (Weisz).} Let $\ k\in \mathbb{P}.$ Then
\begin{equation*}
\left\Vert \sigma _{k}f\right\Vert _{H_{p}}\leq c_{p}\left\Vert f\right\Vert
_{H_{p}},\ \text{\ }\left( \text{\ }f\in H_{p},\text{ \ \ }p>1/2\right)
\end{equation*}%
and
\begin{equation*}
\left\Vert \sigma _{M_{k}}f\right\Vert _{H_{p}}\leq c_{p}\left\Vert
f\right\Vert _{H_{p}},\ \text{\ }\left( \text{\ }f\in H_{p},\text{ \ \ }%
p>1/2\right)
\end{equation*}

For the Walsh system Goginava \cite{gog1} proved a very unexpected fact:

\textbf{Theorem G1 (Goginava).} Let\textbf{\ }$0<p\leq 1.$\textbf{\ }Then%
\textbf{\ }there exists a martingale $f\in H_{p},$ such that $\ $%
\begin{equation*}
\sup_{n\in \mathbb{P}}\left\Vert \left\vert \sigma _{M_{k}}f\right\vert
\right\Vert _{H_{p}}=+\infty ,\text{ \ }\left( 0<p\leq 1/2\right) .
\end{equation*}

In \cite{tep1} (see also \cite{GoAMH}) it was proved that there exists a
martingale $f\in H_{p},$ such that $\ $%
\begin{equation*}
\sup_{n\in \mathbb{P}}\left\Vert \sigma _{n}f\right\Vert _{H_{p}}=+\infty ,%
\text{ \ }\left( 0<p\leq 1/2\right) .
\end{equation*}

In \cite{tep18} it was proved that the following statements are true:

\textbf{Theorem T2 (Tephnadze). }a)\textbf{\ }Let $0<p\leq 1/2,$ $f\in
H_{p}, $ $M_{N}<n\leq $ $M_{N+1}$ and
\begin{equation}
\omega _{H_{p}}\left( 1/M_{N},f\right) =o\left( 1/M_{N}^{1/p-2}N^{2\left[
1/2+p\right] }\right) ,\text{ as \ }N\rightarrow \infty .  \label{cond}
\end{equation}%
Then
\begin{equation*}
\left\Vert \sigma _{n}f-f\right\Vert _{p}\rightarrow 0,\text{ when }%
n\rightarrow \infty .
\end{equation*}

b) Let $0<p<1/2$ and $M_{N}<n\leq M_{N+1}.$ Then there exists a martingale $%
f\in H_{p}(G_{m}),$\ \ for which
\begin{equation}
\omega _{H_{p}}\left( 1/M_{N},f\right) =O\left( 1/M_{N}^{1/p-2}\right) ,%
\text{ \ as \ }N\rightarrow \infty  \label{cond2}
\end{equation}%
and
\begin{equation*}
\left\Vert \sigma _{n}f-f\right\Vert _{L_{p,\infty }}\nrightarrow 0,\,\,\,%
\text{as\thinspace \thinspace \thinspace }n\rightarrow \infty .
\end{equation*}%
c) Let $M_{N}<n\leq M_{N+1}.$ Then there exists a martingale $f\in
H_{1/2}(G_{m}),$\ \ for which
\begin{equation}
\omega _{H_{1/2}}\left( 1/M_{N},f\right) =O\left( 1/N^{2}\right) ,\text{ \
as \ }N\rightarrow \infty  \label{cond3}
\end{equation}%
and
\begin{equation*}
\left\Vert \sigma _{n}f-f\right\Vert _{1/2}\nrightarrow 0,\,\,\,\text{%
as\thinspace \thinspace \thinspace }n\rightarrow \infty .
\end{equation*}

In this paper we will show that Theorem W3 of Weisz are simple corollary of
Theorems W1 and W2. It is very important, because we do not have definition
of conjugate transform of martingales, with the same properties as Walsh
series. Moreover we will improve inequality (\ref{1}) and show that
\begin{equation*}
\left\Vert \sigma _{n}f\right\Vert _{H_{p}}\leq c_{p}n^{1/p-2}\log ^{2\left[
1/2+p\right] }n\left\Vert f\right\Vert _{H_{p}},\text{ \ \ }\left(
n=2,3,....\right) .
\end{equation*}

On the other hand, it gives chance to generalize Theorem T2 and derive
necessary and sufficient conditions for the convergence of Fejér means in
terms of the modulus of continuity of the Hardy spaces $H_{p},$ $\left(
0<p\leq 1\right) $. We will also generalize Theorem G1 for the bounded
Vilenkin system.

\section{The main result}

\begin{theorem}
a) Let $f\in H_{p},$ where $1/2<p\leq 1.$ Then%
\begin{equation*}
\left\Vert \sigma _{n}f\right\Vert _{H_{p}}\leq c_{p}\left\Vert f\right\Vert
_{H_{p}},\text{ }\left( n\in \mathbb{P}\right) .
\end{equation*}
\end{theorem}

\textit{b) Let } $f\in H_{p},$ \textit{where} $0<p\leq 1/2.$ \textit{Then}%
\begin{equation*}
\left\Vert \sigma _{n}f\right\Vert _{H_{p}}\leq c_{p}n^{1/p-2}\log ^{2\left[
1/2+p\right] }n\left\Vert f\right\Vert _{H_{p}},\text{ \ }\left( n\in
\mathbb{P}\right) .
\end{equation*}%
\textit{c) Let } $f\in H_{p},$ \textit{where} $p>0.$ \textit{Then}%
\begin{equation*}
\left\Vert \sigma _{M_{n}}f\right\Vert _{H_{p}}\leq c_{p}\left\Vert
f\right\Vert _{H_{p}},\text{ \ }\left( n\in \mathbb{P}\right) .
\end{equation*}%
\textit{d)} \textit{Let} $p>1/2$ \textit{and} $f\in H_{p}.$ \textit{Then}%
\begin{equation*}
\left\Vert \sigma _{n}f-f\right\Vert _{H_{p}}\rightarrow 0,\text{ when }%
n\rightarrow \infty .
\end{equation*}

\textit{e) \ Let} $0<p\leq 1/2,$ $f\in H_{p},$ $M_{N}<n\leq $ $M_{N+1}$
\textit{and}
\begin{equation}
\omega _{H_{p}}\left( 1/M_{N},f\right) =o\left( 1/M_{N}^{1/p-2}N^{2\left[
1/2+p\right] }\right) ,\text{ as \ }N\rightarrow \infty .  \label{cond200}
\end{equation}%
\textit{Then}
\begin{equation*}
\left\Vert \sigma _{n}f-f\right\Vert _{H_{p}}\rightarrow 0,\text{ when }%
n\rightarrow \infty .
\end{equation*}

\begin{proof}[Proof of Theorem 1.]
Let $f\in H_{p},$ $p>1$ and $M_{N}\leq n<M_{N+1}.$ Then \
\begin{equation}
E\sigma _{n}f:=\left( S_{M_{k}}\sigma _{n}f:k\geq 0\right) =\left( \frac{%
M_{0}}{n}\sigma _{M_{0}}f,...,\frac{M_{N}}{n}\sigma _{M_{N}}f,\text{ }\sigma
_{n}f\right)  \label{fe11}
\end{equation}%
and
\begin{equation*}
\left( E\sigma _{n}f\right) ^{\ast }\leq \sup_{0\leq k\leq N}\left\vert
\frac{M_{k}}{n}\sigma _{M_{k}}f\right\vert +\left\vert \sigma
_{n}f\right\vert \leq \sigma ^{\#}f+\left\vert \sigma _{n}f\right\vert .
\end{equation*}%
By combining (\ref{maxfn}) and (\ref{1}) we get that%
\begin{eqnarray}
\left\Vert \sigma _{n}f\right\Vert _{H_{p}} &:&=\left\Vert \left( E\sigma
_{n}f\right) ^{\ast }\right\Vert _{p}\leq \left\Vert \sigma
^{\#}f\right\Vert _{p}+\left\Vert \sigma _{n}f\right\Vert _{p}  \label{fe12}
\\
&\leq &c_{p}\left\Vert f\right\Vert _{H_{p}},\text{ \ }\left( 1/2<p\leq
1\right)  \notag
\end{eqnarray}%
and%
\begin{eqnarray}
\left\Vert \sigma _{n}f\right\Vert _{H_{p}} &:&=\left\Vert \left( \sigma
_{n}f\right) ^{\ast }\right\Vert _{p}\leq \left\Vert \sup_{k\in
_{+}}\left\vert \sigma _{M_{k}}f\right\vert \right\Vert _{p}+\left\Vert
\sigma _{n}f\right\Vert _{p}  \label{stm1} \\
&\leq &c_{p}\left( n^{1/p-2}\log ^{2\left[ 1/2+p\right] }n\right) \left\Vert
f\right\Vert _{H_{p}},\text{ \ }\left( 0<p\leq 1/2\right) .  \notag
\end{eqnarray}

On the other hand, if $n=M_{N},$ for some $n\in \mathbb{P},$ by using (\ref%
{fe11}), we obtain that
\begin{equation*}
\left( E\sigma _{M_{N}}f\right) ^{\ast }\leq \sup_{0\leq k\leq N}\left\vert
\frac{M_{k}}{n}\sigma _{M_{k}}f\right\vert \leq \sup_{k\in \mathbb{N}%
_{+}}\left\vert \sigma _{M_{k}}f\right\vert =:\sigma ^{\#}f
\end{equation*}%
and%
\begin{eqnarray}
\left\Vert \sigma _{M_{N}}f\right\Vert _{H_{p}} &:&=\left\Vert \left(
E\sigma _{M_{N}}f\right) ^{\ast }\right\Vert _{p}\leq \left\Vert \sigma
^{\#}f\right\Vert _{p}  \label{e5} \\
&\leq &c_{p}\left\Vert f\right\Vert _{H_{p}},\text{ \ }\left( p>0\right) .
\notag
\end{eqnarray}

It is easy to show that (see \cite{tep18})%
\begin{equation}
\sigma _{n}S_{M_{N}}f-S_{M_{N}}f=\frac{M_{N}}{n}S_{M_{N}}\left( \sigma
_{M_{N}}f-f\right) .  \label{mc1}
\end{equation}

Hence, according to (\ref{mc1}), we have that%
\begin{eqnarray*}
&&\left\Vert \sigma _{n}f-f\right\Vert _{H_{p}} \\
&\leq &c_{p}\left\Vert \sigma _{n}f-\sigma _{n}S_{M_{N}}f\right\Vert
_{H_{p}}+c_{p}\left\Vert \sigma _{n}S_{M_{N}}f-S_{M_{N}}f\right\Vert
_{H_{p}}+c_{p}\left\Vert S_{M_{N}}f-f\right\Vert _{H_{p}} \\
&=&c_{p}\left\Vert \sigma _{n}\left( S_{M_{N}}f-f\right) \right\Vert
_{H_{p}}+c_{p}\left\Vert S_{M_{N}}f-f\right\Vert _{H_{p}}+\frac{c_{p}M_{N}}{n%
}\left\Vert S_{M_{N}}\sigma _{M_{N}}f-S_{M_{N}}f\right\Vert _{H_{p}} \\
&:&=III+IV+V.
\end{eqnarray*}

For $IV$ we have that
\begin{equation*}
IV=c_{p}\omega _{H_{p}}\left( 1/M_{n},f\right) \rightarrow 0,\text{ \ \ \ as
\ \ \ }n\rightarrow \infty ,\text{ \ \ }\left( p>0\right) .
\end{equation*}%
Since \
\begin{equation}
\left\Vert S_{M_{n}}f\right\Vert _{H_{p}}\leq c_{p}\left\Vert f\right\Vert
_{H_{p}},\ \ p>0  \label{sn1}
\end{equation}%
\ we obtain that%
\begin{equation*}
V\leq \left\Vert S_{M_{N}}\left( \sigma _{M_{N}}f-f\right) \right\Vert
_{H_{p}}\leq \left\Vert \sigma _{M_{N}}f-f\right\Vert _{H_{p}}\rightarrow 0,%
\text{ \ \ \ as \ \ \ \ }n\rightarrow \infty .
\end{equation*}

Let $1/2<p\leq 1.$ Then, by using (\ref{fe12}) we obtain that%
\begin{equation*}
III\leq c_{p}\left\Vert S_{M_{N}}f-f\right\Vert _{H_{p}}\leq c_{p}\omega
_{H_{p}}\left( 1/M_{N},f\right) \rightarrow \infty ,\text{ \ \ as \ \ }%
n\rightarrow \infty .
\end{equation*}

On the other hand, for $0<p\leq 1/2$ we can apply (\ref{stm1}) and under
condition (\ref{cond200}) we get that
\begin{equation*}
III\leq c_{p}\left( n^{1/p-2}\log ^{2\left[ 1/2+p\right] }n\right) \omega
_{H_{p}}\left( 1/M_{N},f\right) \rightarrow 0,\text{ \ \ as \ \ }%
n\rightarrow \infty .
\end{equation*}

The proof is complete.
\end{proof}

\begin{theorem}
Let $0<p\leq 1.$ Then the operator $\left\vert \sigma _{M_{n}}f\right\vert $
is not bounded from the martingale Hardy space $H_{p}$ $\left( G_{m}\right) $
to the martingale Hardy space $H_{p}$ $\left( G_{m}\right) .$
\end{theorem}

\begin{proof}[Proof of Theorem 2.]
Let%
\begin{equation*}
f_{A}=D_{M_{A+1}}-D_{M_{_{A}}}.
\end{equation*}

It is evident that
\begin{equation*}
\widehat{f}_{A}\left( i\right) =\left\{
\begin{array}{l}
\text{ }1,\text{ if }i=M_{_{A}},...,M_{A+1}-1, \\
\text{ }0,\text{otherwise}.%
\end{array}%
\right.
\end{equation*}%
Then we can write
\begin{equation}
S_{i}f_{A}=\left\{
\begin{array}{l}
D_{i}-D_{M_{_{A}}},\text{ \ \ \ \ if }i=M_{A},...,M_{A+1}-1, \\
\text{ }f_{A},\text{ \ \ \ \ \ \ \ \ \ \ \ \ \ \ \ if }i\geq M_{A+1}, \\
0,\text{ \qquad\ \ \ \ \ \ \ \ \ \ otherwise}.%
\end{array}%
\right.  \label{4}
\end{equation}

From (\ref{2}) we get that (c.f. \cite{tep2} and \cite{tep3})
\begin{equation}
\left\Vert f_{A}\right\Vert _{H_{p}}=\left\Vert \sup\limits_{n\in \mathbb{P}%
}S_{M_{n}}\left( f_{A}\right) \right\Vert _{p}=\left\Vert f_{A}\right\Vert
_{p}\leq M_{A}^{1-1/p}.  \label{111}
\end{equation}%
Let $x\in I_{A+1}.$ Applying (\ref{4}), we obtain that
\begin{eqnarray}
&&\sigma _{M_{A+1}}f_{A}\left( x\right) =\frac{1}{M_{A+1}}\overset{M_{A+1}}{%
\underset{j=0}{\sum }}S_{j}f_{A}\left( x\right) =\frac{1}{M_{A+1}}\overset{%
M_{A+1}}{\underset{j=M_{A}+1}{\sum }}S_{j}f_{A}\left( x\right)  \label{3} \\
&=&\frac{1}{M_{A+1}}\overset{M_{A+1}}{\underset{j=M_{A}}{\sum }}\left(
D_{_{j}}\left( x\right) -D_{M_{A}}\left( x\right) \right) =\frac{1}{M_{A+1}}%
\overset{M_{A+1}}{\underset{j=M_{A}}{\sum }}\left( j-M_{A}\right)  \notag \\
&=&\frac{1}{M_{A+1}}\overset{\left( m_{A}-1\right) M_{A}}{\underset{j=0}{%
\sum }}j\geq cM_{A}.  \notag
\end{eqnarray}%
By using (\ref{3}), we find that%
\begin{eqnarray}
S_{M_{N}}\left( \left\vert \sigma _{M_{A+1}}f_{A}\right\vert ;x\right)
&=&\int_{G_{m}}\left\vert \sigma _{M_{A+1}}f_{A}\left( t\right) \right\vert
D_{M_{N}}\left( x-t\right) d\mu \left( t\right)  \label{5} \\
&\geq &\int_{I_{A+1}}\left\vert \sigma _{M_{A+1}}f_{A}\left( x\right)
\right\vert D_{M_{N}}\left( x-t\right) d\mu \left( t\right)  \notag \\
&\geq &cM_{A}\int_{I_{A+1}}D_{M_{N}}\left( x-t\right) d\mu \left( t\right) .
\notag
\end{eqnarray}

According to (\ref{5}), we have that%
\begin{equation*}
S_{M_{N}}\left( \left\vert \sigma _{M_{A+1}}f_{A}\right\vert ;x\right) \geq
cD_{M_{N}}\left( x\right) ,\text{ \ }N=0,1,...,A,
\end{equation*}%
and%
\begin{equation*}
\sup_{N}S_{M_{N}}\left( \left\vert \sigma _{M_{A+1}}f_{A}\right\vert
;x\right) \geq \sup_{1\leq N<A}S_{M_{N}}\left( \left\vert \sigma
_{M_{A+1}}f_{A}\right\vert ;x\right) \geq c\sup_{1\leq N<A}D_{M_{N}}\left(
x\right) .
\end{equation*}%
Let $x\in I_{N}\backslash I_{N+1},$ for some $s=0,1,...,A.$ Then, from (\ref%
{2}) it follows that%
\begin{equation*}
\sup_{N\in \mathbb{P}}S_{M_{N}}\left( \left\vert \sigma
_{M_{A+1}}f_{A}\right\vert ;x\right) \geq cM_{N}.
\end{equation*}

Let $0<p<1.$ Then
\begin{eqnarray}
&&\left\Vert \left\vert \sigma _{M_{A+1}}f_{A}\right\vert \right\Vert
_{H_{p}}^{p}  \label{6} \\
&=&\left\Vert \sup_{1\leq N<A-1}S_{M_{N}}\left( \left\vert \sigma
_{M_{A+1}}f_{A}\right\vert ;x\right) \right\Vert _{p}^{p}  \notag \\
&\geq &\int_{G_{m}}\left( \sup_{1\leq N<A-1}S_{M_{N}}\left( \left\vert
\sigma _{M_{A+1}}f_{A}\right\vert ;x\right) \right) ^{p}d\mu \left( x\right)
\notag \\
&\geq &\overset{A}{\underset{s=1}{\sum }}\int_{I_{N}\backslash
I_{N+1}}\left( \sup_{1\leq N<A-1}S_{M_{N}}\left( \left\vert \sigma
_{M_{A+1}}f_{A}\right\vert ;x\right) \right) ^{p}d\mu \left( x\right)  \notag
\\
&\geq &c\overset{A}{\underset{s=1}{\sum }}\frac{M_{s}^{p}}{M_{s}}=c_{p}>0.
\notag
\end{eqnarray}

Let $p=1.$ Then we obtain that%
\begin{equation}
\left\Vert \left\vert \sigma _{M_{A+1}}f_{A}\right\vert \right\Vert
_{H_{1}}\geq cA.  \label{7}
\end{equation}

By combining (\ref{111}), (\ref{6}) and (\ref{7}) we can conclude that%
\begin{equation*}
\frac{\left\Vert \left\vert \sigma _{M_{A+1}}f_{A}\right\vert \right\Vert
_{H_{p}}}{\left\Vert f_{A}\right\Vert _{H_{p}}}\geq \frac{c_{p}}{%
M_{A}^{1-1/p}}\rightarrow \infty ,\text{ \ \ as \ \ \ }A\rightarrow \infty ,%
\text{ \ \ }0<p<1
\end{equation*}%
and%
\begin{equation*}
\frac{\left\Vert \left\vert \sigma _{M_{A+1}}f_{A}\right\vert \right\Vert
_{H_{1}}}{\left\Vert f_{A}\right\Vert _{H_{1}}}\geq cA\rightarrow \infty ,%
\text{ \ \ as \ \ \ }A\rightarrow \infty .
\end{equation*}

The proof is complete.
\end{proof}

As the consequence of our result we have the following negative result:

\begin{corollary}
Let $0<p\leq 1.$ Then the maximal operator $\sigma ^{\#}f$ $\ $is not
bounded from the martingale Hardy space $H_{p}$ $\left( G_{m}\right) $ to
the martingale Hardy space $H_{p}$ $\left( G_{m}\right) .$
\end{corollary}

\end{document}